\documentclass[a4paper,12pt]{amsart}
\usepackage{hyperref}
\usepackage{amssymb}
\usepackage{color}

\allowdisplaybreaks

\newtheorem{thm}{Theorem}

\newtheorem{cor}[thm]{Corollary}

\theoremstyle{definition}

\theoremstyle{remark}
\newtheorem{rmk}[thm]{Remark}
\newtheorem{ex}[thm]{Example}

\newcommand{\C}{\mathbb{C}}
\newcommand{\N}{\mathbb{N}}

\newcommand{\T}{\mathbb{T}}
\newcommand{\Z}{\mathbb{Z}}
\newcommand{\subr}{\operatorname{r}}

\newcommand{\id}{\operatorname{id}}

\newcommand{\Aut}{\operatorname{Aut}}

\title[On graded $C^*$-algebras]{\boldmath{On graded $C^*$-algebras}}
\author[Iain Raeburn]{Iain Raeburn}

\address{Iain Raeburn\\ Department of Mathematics and Statistics\\University of Otago\\PO Box 56\\Dunedin 9054\\New Zealand; orcid id 0000-0002-0368-5979}
\address{Current address: Department of Mathematics and Statistics\\Victoria University of Wellington\\PO Box 600\\Wellington 6140\\New Zealand}
\email{iraeburn@maths.otago.ac.nz; iain.raeburn@vuw.ac.nz}
\subjclass[2010]{46L05}
\date{\today}

\begin{document}

\begin{abstract}
We show that every topological grading of a $C^*$-algebra by a discrete abelian group is implemented by an action of the compact dual group.
\end{abstract}

\maketitle

Suppose that $A$ is an algebra over a field $K$ and $G$ is a group. We say that $A$ is \emph{$G$-graded} if there are linear subspaces $\{A_g:g\in G\}$ such that $A$ is the direct sum of the $A_g$ and $a\in A_g$, $b\in A_h$ imply $ab\in A_{gh}$.  Then each element of $A$ has a unique decomposition as a sum $a=\sum_{g\in G}a_g$ of homogeneous components $a_g\in A_g$ (and all but finitely many $a_g=0$). We have known since the first paper on the subject that the Leavitt path algebras $L_K(E)$ of a directed graph $E$ are $\Z$-graded  \cite[Lemma~1.7]{AAP}.

For graph $C^*$-algebras, the field $K$ is always $\C$. The graph algebra $C^*(E)$ is not graded in the algebraic sense, and the role of the grading in the general  theory is played by a \emph{gauge action} $\gamma$ of the circle $\T=\{z\in \C:|z|=1\}$ on $C^*(E)$. We can use this action to define homogeneous components of $a\in C^*(E)$ by
\[
a_n:=\int_0^1 \gamma_{e^{2\pi it}}(a)e^{-2\pi int}\,dt\quad\text{for $n\in \Z$.}
\]
But $\{n:a_n\not=0\}$ can be infinite, and then the relationship between $a$ and the sequence $\{a_n\}$ is well-known to be analytically subtle (see \cite{Rmalaga}, for example).

In the recent book \cite{AAMS}, the authors show that $C^*(E)$ is always graded in a weaker sense introduced by Exel \cite{E}. He said that a $C^*$-algebra $A$ is \emph{$G$-graded} if there is a family $\{A_g:g\in G\}$ of linearly independent closed subspaces such that $a\in A_g$ and $b\in A_h$ imply $ab\in A_{gh}$ and $a^*\in A_{g^{-1}}$, and such that $A$ is the norm-closure of $\bigoplus_{g\in G}A_g$. It is proved in \cite[Proposition~5.2.11]{AAMS} that every graph algebra $C^*(E)$ is $\Z$-graded in Exel's sense. 

In fact the result in \cite{AAMS} says rather more than this. Exel also introduced a stronger notion: he said that a $G$-graded $C^*$-algebra $A$ is \emph{topologically graded} if there is a bounded linear map $F:A\to A$ which is the identity on $A_e$ and vanishes on every $A_g$ with $g\not=e$ \cite[\S19]{Ebook}. (His original Definition~3.4 in \cite{E} looks a little stronger, since it asserts that $F$ is a conditional expectation. But it follows from \cite[Theorem~3.3]{E} or \cite[Theorem~19.1]{Ebook} that this extra requirement is automatic.) The extra information in 
\cite[Proposition~5.2.11]{AAMS} implies that the $\Z$-grading of $C^*(E)$ is topological, and this information is obtained using the gauge action of $\T$.
 
Here we revisit gradings of $C^*$-algebras. We work with topological gradings by an abelian group $G$, because that is enough to cover also the $\Z^k$-graded graph algebras of higher-rank graphs and their twisted analogues. We show that every topological $G$-grading of a $C^*$-algebra $A$ is implemented by a natural action of the Pontryagin dual $\widehat G$, which in the case of a graph algebra $C^*(E)$ is the usual gauge action of $\T=\widehat \Z$. We then use recent results on the $C^*$-algebras of Fell bundles \cite{R2} to reconstruct an arbitrary element of a topologically $\Z^k$-graded algebra from its graded components. 
\medskip

We begin by discussing a couple of illustrative examples from Exel's book \cite{Ebook}. 

\begin{ex}\label{circle}
We consider the graph $E$ with one vertex $v$ and one loop $e$. The graph algebra $C^*(E)$ has identity $P_v$ and is generated by the unitary element $S_e$. Because graph algebras are universal for Cuntz-Krieger families, this one is universal for $C^*$-algebras generated by a unitary element, and hence $(C^*(E),S_e)$ is $(C(\T),z)$. The gauge action of $\T$ is implemented by rotations, and for $n\in \Z$, the graded components $C^*(\T)_n$ are the scalar multiples of the polynomials $z^n$. Since these polynomials form an orthonormal basis for $L^2(\T)$ and $C(\T)\subset L^2(\T)$, the Fourier coefficients
\[
\widehat f(n)=\int_{\T} f(z)z^{-n}\,dz:=\int_0^1 f(e^{2\pi it})e^{-2\pi int}\,dt
\]
of $f\in C(\T)$ determine $f$ uniquely: $\widehat f(n)=\widehat g(n)$ for all $n$ implies $f=g$ in $C(\T)$. It has long been known that the Fourier series of $f$ need not converge in the norm of the ambient $C^*$-algebra $C(\T)$, but a classical  theorem of F\'ejer (1900) tells us that the C\'esaro means of the partial sums of the Fourier series converge uniformly to $f$ on $\T$. Thus we can recover $f$ from its Fourier coefficients, and, provided we remember that this recovery process is not the obvious one, we can view $C(\T)$ as a $\Z$-graded algebra.
\end{ex}

\begin{ex}\label{Exelex}
(Motivated by the discussion following Proposition~19.3 in \cite{Ebook}.) We take a closed subset $X$ of $\T$ which is infinite but not all of of $\T$, and consider $C(X)$. We write $e_n$ for the polynomial $z^n$, viewed as an element of $C(\T)$. Then, as observed in \cite{Ebook}, the subspaces
\[
C(X)_n:=\{ce_n|_X:c\in \C\}
\]
are linearly independent (because $X$ is infinite). Because the $e_n$ span a dense subspace of $C(\T)$ and $f\mapsto f|_X$ is a surjection of $C(\T)$ onto $C(X)$, the direct sum $\bigoplus_{n\in \Z}C(X)_n$ is dense in $C(X)$. Thus the $C(X)_n$ give a $\Z$-grading of $C(X)$ in the sense of \cite{E,Ebook}.

Since $X$ is a proper closed subset of $\T$, and the map $f\mapsto f|_X$ has infinite-dimensional kernel isomorphic to $C_0(\T\setminus X)$, each $f\in C(X)$ has many extensions $g$ in $C(\T)$. Each such extension $g$ has a canonical sequence of homogeneous components $\widehat g(n)e_n$, and the C\'esaro means for this sequence converge uniformly in $C(\T)$ to $g$. The restrictions of the C\'esaro means to $X$ converge uniformly in $C(X)$ to $g|_X=f$. But different extensions of $f$ have different Fourier coefficients, and hence there is no canonical choice of homogeneous components for $f$ in $C(X)$.
\end{ex}

Example~\ref{Exelex} shows that a $\Z$-graded $C^*$-algebra need not have the properties one would expect of a grading. So Exel also considered his stronger notion of ``topological grading'', in which the bounded linear map $F:A\to A_e$ which gives a continuous choice of homogeneous component $a_e:=F(a)$. In the discussion in \cite[\S19]{Ebook}, he proves that the algebra $C(X)$ in Example~\ref{Exelex} is not topologically graded. Our main result says that for a topologically $G$-graded $C^*$-algebra, the map $F$ is implemented by integration of a continuous action of the compact dual group $\widehat G$ with respect to the normalised Haar measure.

\begin{thm}\label{intform}
Suppose that $G$ is an abelian group and that $A$ is a $C^*$-algebra which is topologically $G$-graded in Exel's sense. Then there is a strongly continuous action $\alpha$ of $\widehat G$ on $A$ such that $\alpha_\gamma(a)=\gamma(g)a$ for $a\in A_g$, and then
\begin{equation}\label{expect=int}
F(a)=\int_{\widehat G} \alpha_\gamma(a)\,d\gamma\quad\text{for all $a\in A$.}
\end{equation}
\end{thm}

The subspaces $\{A_g:g\in G\}$ in the  $G$-grading form a \emph{Fell bundle} $B$ over $G$. There is an extensive theory of Fell bundles, orginally developed by Fell (he called them \emph{$C^*$-algebraic bundles} \cite{FD}), and revisited by several authors in the 1990s. We shall lean heavily on results of Exel \cite{E}, as presented in his recent monograph \cite{Ebook}. 

Each Fell bundle $B$ over a (discrete) group $G$ has an enveloping $C^*$-algebra $C^*(B)$ that is universal for a class of Hilbert-space representations, consisting of linear maps $\pi_g:A_g\to B(H)$ such that  $\pi_g(a)\pi_h(b)=\pi_{gh}(ab)$, $\pi_g(a)^*=\pi_{g^{-1}}(a^*)$, and such that $\pi_e$ is a nondegenerate representation of $A_e$. There is also a reduced $C^*$-algebra $C^*_{\subr}(B)$ which is generated by a regular representation \cite[\S17]{Ebook}. Because we are interested in Fell bundles over abelian groups, all our Fell bundles are amenable in Exel's sense \cite[Theorem~20.7]{Ebook}, and $C^*(B)=C_{\subr}^*(B)$.

\begin{ex}
A $G$-graded algebra can be quite different from the $C^*$-algebra of its Fell bundle. To see this, consider the Fell bundles $B_1$ and $B_2$ over $\Z$ associated to the gradings of $C(\T)$ in Example~\ref{circle} and $C(X)$ in Example~\ref{Exelex}. The maps $ce_n\mapsto ce_n|_X$ are Banach-space isomorphisms of the fibres $B_{1,n}$ onto the fibres $B_{2,n}$ (both are one-dimensional) and respect the Fell-bundle structure. Since $C(\T)$ is topologically graded (on any graph algebra there is a map $a\mapsto a_0$ defined by averaging over the gauge action), we have $C^*(B_1)=C(\T)$. Thus we also have $C^*(B_2)=C(\T)$.
\end{ex}

\begin{proof}[Proof of Theorem~\ref{intform}]
Because $A$ is topologically graded there is a bounded linear map $F:A\to A$ such that $f(a)=a$ for $a\in A_e$ and $f(a)=0$ for $a\in A_g$ with $g\not=e$. Let $B$ be the corresponding Fell bundle over $G$ with fibres $A_g$. We deduce from \cite[Theorem~19.5]{Ebook} that there are  surjections $\phi$ of $C^*(B)$ onto $A$ and $\psi$ of $A$ onto the reduced algebra $C^*_{\subr}(B)$ such that $\psi \circ\phi$ is the regular representation of $C^*(B)$. Since the group $G$ is abelian, the Fell bundle is amenable, and the regular representation is an isomorphism. Hence so are $\phi$ and $\psi$.  We deduce that $A$ is generated by a representation $\rho$ of $B$ in $A$, and that $(A,\rho)$ is universal for Hilbert-space representations of $B$.

We now fix $\gamma\in \widehat G$. For each $g\in G$, we define $\alpha_{\gamma,g}:A_g\to A$ by $\alpha_{\gamma,g}(a)=\gamma(g)a$. Since $|\gamma(g)|=1$, $\alpha_{\gamma,g}$ is a linear and isometric embedding of the Banach space $A_g$ in $A$. Since each $A_g$ is a left Hilbert module over $A_e$, the action of $A_e$ on $A_g$ is nondegenerate \cite[Corollary~2.7]{tfb}, and since $A=\overline{\bigoplus_g A_g}$, it follows that any approximate identity for $A_e$ is also an approximate identity for $A$. Thus $\alpha_{\gamma, e}$ is nondegenerate. For $a\in A_g$, $b\in A_h$ we have
\begin{align*}
&\alpha_{\gamma,g}(a)\alpha_{\gamma,h}(b)=(\gamma(g)a)(\gamma(h)b)=\gamma(gh)ab=\alpha_{\gamma,gh}(ab),\text{ and}\\
&u_{\gamma,g}(a)^*=(\gamma(g)a)^*=\overline{\gamma(g)}a^*=\gamma(g^{-1})a^*=\alpha_{\gamma,g^{-1}}(a^*).
\end{align*}
Thus $\alpha_\gamma=\{\alpha_{\gamma, g}\}$ is a representation of the Fell bundle $B$, and the universal property of $A=C^*(B)$ gives a nondegenerate homomorphism $\alpha_\gamma:A\to A$ such that $\alpha_\gamma\circ\rho_g =\alpha_{\gamma,g}$ for $g\in G$. 

For $\gamma,\chi\in \widehat G$ we have $\alpha_\gamma\alpha_\chi=\alpha_{\gamma\chi}$ on each $A_g$, and hence also on $A=\overline{\bigoplus A_g}$. Since $\alpha_1$ is the identity on $A$, it follows that each $\alpha_\gamma$ is an isomorphism, and that $\gamma\mapsto \alpha_\gamma$ is a homomorphism of $\widehat G$ into the automorphism group $\Aut A$. Since convergence in the dual of a discrete abelian group is pointwise convergence, the map $\gamma\mapsto \alpha_\gamma(a)$ is continuous for each $a\in A_g$, and hence by an $\epsilon/3$ argument for all $a\in A=\overline{\bigoplus A_g}$. Thus $\alpha$ is a strongly continuous action of $\widehat G$ on $A$.

Now averaging with respect to normalised Haar measure on $\widehat G$ gives a conditional expectation $E$ of $A$ onto the fixed-point algebra $A^\alpha$ such that
\[
E(a)=\int_{\widehat G} \alpha_\gamma(a)\,d\gamma\quad\text{for all $a\in A$}
\]
(following the discussion for $\widehat G=\T$ in the first few pages of \cite[Chapter~3]{R}, for example). Since $\alpha_\gamma(a)=a$ for $a\in A_e$ and we are using the normalised Haar measure, we have $E(a)=a$ for $a\in A_e$. For $a\in A_g$ with $g\not= e$, we have
\[
E(a)=\int_{\widehat G} \gamma(g)a\,d\gamma=\Big(\int_{\widehat G} \gamma(g)\,d\gamma\Big)a=0.
\]
Thus $E=F$ on $\bigoplus A_g$, and hence by continuity of $E$ and $F$ also on the closure $A$.
\end{proof}

Since $E$ is a faithful conditional expectation, we deduce that $F$ is too. Hence:

\begin{cor}
The bounded linear map $F:A\to A_e$ in Theorem~\ref{intform} is a conditional expectation onto $A_e$, and is faithful in the sense that $F(a^*a)=0$ implies $a=0$.
\end{cor}

As we remarked earlier, Exel also proved directly in \cite{E} that $F$ is a conditional expectation.

\begin{rmk}
We have concentrated on Fell bundles over abelian groups because our motivation for looking at this material came from graph algebras, where the appropriate group is $G=\Z^k$. However, the first paragraph of the proof of Theorem~\ref{intform} works for arbitrary amenable groups. Then we can use the universal property of $C^*(B)$ to construct a coaction $\delta:A\to A\otimes C^*(G)$ such that $\delta(a)=a\otimes u_g$ for $a\in A_g$ (see the preliminary material in \cite[Appendix~B]{R2}). The group algebra $C^*(G)$ has a trace $\tau$ characterised by $\tau(1)=1$ and $\tau(u_g)=0$ for $g\not=e$, and hence there is a slice map ${\id}\otimes \tau: A\otimes C^*(G)\to A$. Composing gives a contraction $E:=({\id}\otimes {\tau})\circ \delta$ of $A$ onto 
\[
A^\delta:=\{a\in A:\delta(a)=a\otimes 1\}.
\]
Again we have $A^\delta=A_e$, and $E=F$. 

When $G$ is not amenable, Theorem~19.5 of \cite{Ebook} only tells us that $A$ lies somewhere between $C^*(B)$ and $C^*_{\subr}(B)$. For $A=C^*(B)$, we can use the coaction of the previous paragraph. If $A=C^*_{\subr}(B)$, then we can use spatial arguments to construct a reduced coaction on $A$ (see \cite[Example~2.3(6)]{LPRS} and \cite{Q}). But in general, trying to construct suitable coactions on $A$ seems likely to pose rather delicate problems in nonabelian duality.
\end{rmk}

We now return to the case of an abelian group $G$ and the set-up of Theorem~\ref{intform}. The action $\alpha:\widehat G\to \Aut A$ allows us to construct homogeneous components 
\[
a_g:=\int_{\widehat G} \alpha_\gamma(a)\overline{\gamma(g)}\,d\gamma\quad\text{ for $a\in A$ and $g\in G$.}
\]
For $a\in A_h$, we have 
\[
a_g=\int_{\widehat G} \alpha_h(a)\overline{\gamma(g)}\,d\gamma\int_{\widehat G} a\,d\gamma =\int_{\widehat G}\gamma(hg^{-1})a\,d\gamma=
\begin{cases}a&\text{if $g=h$}\\
0&\text{if $g\not=h$.}
\end{cases}
\]
Comparing this with the formula in \cite[Corollary~19.6]{Ebook}, we see that $a_g$ is the same as Exel's Fourier coefficient $F_g(a)$.

Our motivation came from applications to graph algebras, and hence we are particularly interested in $\Z^k$-graded $C^*$-algebras. Besides the usual graph algebras of directed graphs, for which $k=1$, this includes the higher-rank graph algebras of \cite{KP} and the  twisted higher-rank graph algebras of \cite{KPS, KPS2} (which by \cite[Corollary~4.9]{R2} can be realised as the $C^*$-algebras of Fell bundles over $\Z^k$). For all these graph algebras, the action of the dual $\T^k$ given by Theorem~\ref{intform} is the usual gauge action.

When $G=\Z^k$, the dual is $\T^k$, and Theorem~\ref{intform} gives us an action $\alpha$ of $\T^k$ on $A$. We then define the homogeneous components of $a\in A$ by
\begin{equation}\label{defFc}
a_n=
\int_{\T^k} \alpha_z(a)z^{-n}\,dz\quad\text{for $n\in \Z^k$}.
\end{equation}
Now Proposition~B.1 of \cite{R2} tells us how to recover $a$ from its homogeneous components $a_n$. More precisely:

\begin{cor}
Suppose that a $C^*$-algebra $A$ is $\Z^k$-graded in Exel's sense. Suppose also that there is a bounded linear map $F:A\to A_e$ such that $F|_{A_g}=0$ for $g\not=e$ and $F|_{A_e}$ is the identity. For $a\in A$ and $n\in \Z^k$, define the homogeneous components $a_n$ using \eqref{defFc}. For $m,n\in \Z^k$, we write $m\leq n$ to mean $n-m\in \N^k$, and set
\begin{align*}
s_n(a)&:=\sum_{-n\leq m \leq n}a_m\text{ for $n\in \N^k$, and }\\
 \sigma_N(a)&:=\frac{1}{\textstyle{\prod_{j=1}^k(N_j+1)}}\sum_{0\leq n\leq N} s_n(a)\text{ for $N\in \N^k$.}
\end{align*}
Then $\|\sigma_N(a)-a\|\to 0$ as $N\to \infty$ in $\N^k$.
\end{cor}

\end{document}